\documentclass[leqno,twoside, 12pt]{amsart}
\usepackage{amssymb, latexsym, esint}
\usepackage{color}

\usepackage{amsthm}
\usepackage{amssymb,amsfonts}
\usepackage[all,arc]{xy}
\usepackage{enumerate}
\usepackage{mathrsfs}
\usepackage{amsmath}
\usepackage{verbatim}%comentarios
\usepackage{hyperref}

\def\R{\mathbb R}
\def\N{\mathbb N}
\def\ve{\varepsilon}
\def\G{\mathcal G}
\def\H{\mathcal H}
\def\P{\mathcal P}

\theoremstyle{plain}
\numberwithin{equation}{section} \numberwithin{figure}{section}

\newtheorem{theorem}{Theorem}[section]
\newtheorem{lemma}[theorem]{Lemma}

\newtheorem{definition}[theorem]{Definition}
  \usepackage{epsfig} %For pictures: screened artwork should be set up with an 85 or 100 line screen
\usepackage{graphicx}  \usepackage{epstopdf}
\theoremstyle{definition}
\newtheorem{remark}[theorem]{Remark}
\newtheorem{example}[theorem]{Example}

\numberwithin{equation}{section}

\begin{document}

%%%%%%%%%%%%%%%%%%%%% Publisher's Area please ignore %%%%%%%%%%%%%%%
%

%
%%%%%%%%%%%%%%%%%%%%%%%%%%%%%%%%%%%%%%%%%%%%%%%%%%%%%%%%%%%%%%%%%%%%

\title[Generalized polyharmonic operator in Orlicz-Sobolev spaces]{An eigenvalue problem for a generalized polyharmonic operator in Orlicz-Sobolev spaces without the $\Delta_2$-condition}

\author{Ignacio Ceresa Dussel, Juli\'an Fern\'andez Bonder and Pablo Ochoa}

\address[I.Ceresa Dussel]{Instituto de C\'alculo -- CONICET and
Departamento de Matem\'atica, FCEN -- Universidad de Buenos Aires. 
Ciudad Universitaria, Edificio $0+\infty$, C1428EGA, Av. Cantilo s/n
Buenos Aires, Argentina\\
{\tt iceresad@dm.uba.ar}}

\address[J.F. Bonder]{Instituto de C\'alculo -- CONICET and
Departamento de Matem\'atica, FCEN -- Universidad de Buenos Aires. 
Ciudad Universitaria, Edificio $0+\infty$, C1428EGA, Av. Cantilo s/n
Buenos Aires, Argentina\\
{\tt jfbonder@dm.uba.ar\\ 
https://mate.dm.uba.ar/\~{}jfbonder/index.html}}

\address[P.Ochoa]{Universidad Nacional de Cuyo-CONICET, Parque Gral. San Mart\'in\\
Mendoza, 5500, Argentina\\
{\tt pablo.ochoa@ingenieria.uncuyo.edu.ar}}


%35J62  	Quasilinear elliptic equations
%46E30  	Spaces of measurable functions ($L^p$-spaces, Orlicz spaces, Köthe function spaces, Lorentz spaces, rearrangement invariant spaces, ideal spaces, etc.)
%35D30      Weak solutions to PDEs
\maketitle

\begin{abstract}
In this paper, we consider a generalized polyharmonic eigenvalue problem of the form  $A(u)= \lambda h(u)$ in a bounded smooth domain with Dirichlet boundary conditions in the setting of higher-order Orlicz-Sobolev spaces. Here, $A$ is a very general operator depending on $u$ and arbitrary higher-order derivatives of $u$, whose growth is governed by an Orlicz function, and $h$ is a lower order term.  Combining the theories of pseudomonotone operators with complementary systems, we prove that this eigenvalue problem has an infinite number of eigenfunctions and that the corresponding sequence of eigenvalues tends to infinite. We point out that the $\Delta_2$-condition is not assumed for the involved Orlicz functions. 
\end{abstract}

\section{Introduction}

%Let $\Omega \subset \R^n$ be an open domain with the segment property. In this work, we study the following generalized nonlinear eigenvalue problem consisting in finding critical values to the following Rayleigh-type quotient
%\begin{equation}\label{general eigenvalue pro}
%\frac{\displaystyle\int_\Omega G(x, u, \nabla u,\dots, \nabla^k u)\, dx}{\displaystyle \int_\Omega H(x, u)\, dx},
%\end{equation}
%where $G$ and $H$ are functions defined in the setting of Orlicz-Sobolev spaces without assuming the $\Delta_2$-condition, and $G$ is strictly related with the Gossez functional \cite{G}.

Eigenvalue problems form a large and well-studied family of problems in the area of partial differential equations,  beginning in the nineteenth century  with the classical  formulation for the Dirichlet Laplacian,
\begin{equation}\label{Problem}
\begin{cases}
-\Delta u = \lambda u &\text{ in }  \Omega\\
u=0 &\text{ in }  \partial \Omega.
\end{cases}
\end{equation}It is well known that the Courant minimax principle guarantees the existence of an infinite sequence of eigenvalues $\lambda_i$ to \eqref{Problem}  tending to $\infty$ as $i\to \infty$. See for instance the books \cite{Taylor} and \cite{Z}. 

In the 1990's, the theory was generalized to deal  with eigenvalue problems involving nonlinear operators  such as  the $p$-Laplacian, 

\begin{equation}\label{Problemp}
\begin{cases}
-\Delta_p u = \lambda |u|^{p-2}u &\text{ in }  \Omega\\
u=0 &\text{ in }  \partial \Omega.
\end{cases}
\end{equation}Many results have been obtained on the structure of the spectrum of \eqref{Problemp}. As for \eqref{Problem}, it was shown in \cite{GA} that \eqref{Problemp} has a sequence of nondecreasing positive eigenvalues converging to infinite. Moreover, the first eigenvalue is simple and isolated \cite{L}. 

Let us also note that recently a lot of effort has been put to extend these results for nonlocal operators (to quote a few works, see for instance \cite{B}, \cite{FP},  \cite{LL} and the references therein).

In all these examples, the corresponding functions typically belong to standard Sobolev spaces. However, many nonlinear phenomena cannot be captured adequately within this framework. In particular, problems with non-standard growth conditions, anisotropic behavior, or rapidly increasing nonlinearities require a more flexible functional setting. Orlicz–Sobolev spaces provide such a setting, allowing the treatment of functionals whose growth is dictated by a general Young function $M(t)$ rather than a fixed power $t^p$.

In this context, the eigenvalue problem has been studied by several authors. See for instance \cite{FN}, \cite{GH}, \cite{GM}, \cite{MLo}, \cite{MT}, \cite{OS}, \cite{T}, \cite{YM}, among others.  What makes these operators especially attractive for applications is their ability to exhibit distinct diffusion regimes depending on the magnitude of $|\nabla u|$; that is, the operator may behave differently when $|\nabla u| \ll 1$ and when $|\nabla u| \gg 1$. This phenomenon is known in the literature as \emph{nonstandard growth} in elliptic operators. In the setting of Orlicz-Sobolev spaces, the classical structure of eigenvalue problems is given by
\begin{equation}\label{m-problem}
\begin{cases}
-\Delta_m u = \lambda h(u) & \text{in } \Omega,\\
u = 0 & \text{in } \partial \Omega,
\end{cases}
\end{equation}
where 
$$
-\Delta_m u := \operatorname{div} \left( \frac{m(|\nabla u|)}{|\nabla u|} \nabla u \right)
$$
denotes a generalized $m$-Laplacian operator and the function $h: \R\to \R$ satisfies certain growth conditions. A central object in the analysis of these problems is the primitive of $m$, defined by 
  $$M(t):=\int_0^t m(s)\,ds.$$The function $M$  is assumed to be an $N$-function.  The natural approach to deal with problem \eqref{m-problem} is to look for critical values of the associated  Rayleigh-type quotient
   $$
   \frac{\displaystyle\int_\Omega M( |\nabla u|)\,dx}{\displaystyle\int_\Omega H(u)\,dx},
   $$
where $H(u)=\int_0^u h(t)\, dt.$

Since this Rayleigh-type quotient is not homogeneous, it is convenient to consider critical values of  $\int_\Omega M( |\nabla u|)\,dx$ subject to the constraint $\int_\Omega H(u)\,dx=\mu$, and these critical values depend on the normalization parameter $\mu$. Moreover, in order to apply the Ljusternik-Schnirelmann theory it is required that the functionals are of class $C^1$ and that the associated Sobolev spaces are reflexive and separable. These conditions impose severe restrictions on $M$, namely, it is needed that $M$ and its conjugate $\bar{M}$  satisfy the $\Delta_2$-condition.

This issue was tackled  by Tienari in \cite{T} applying a suitable Galerkin type argument was able to overcome these difficulties  and proved the existence of a sequence of eigenvalues tending to infinite for each normalization parameter.

\normalcolor

   In  \cite{G}, Gossez considers  higher-order operators of the form
   $$
   A(u)=\sum_{|\alpha|\leq k} (-1)^{|\alpha|} D^{\alpha} A_{\alpha}\bigl(x,u,\ldots,\nabla^{k}u\bigr),
   $$
   where the functions $A_\alpha$ satisfy suitable Carath\'eodory, growth, and monotonicity conditions (see conditions (4.2)--(4.4) in \cite{G}). Under these assumptions, Gossez was able to find conditions under which the source problem $$A(u)=f \quad \text{in }\Omega, \qquad u=0 \quad \text{on }\partial \Omega$$has a weak solution in the appropriate Sobolev space.

% and by working within a complementary system, it is shown that the operator $T$, defined on an appropriate domain by
%   $$
%   \langle v, Tu \rangle = \int_{\Omega} \sum_{|\alpha|\leq k} A_{\alpha}\bigl(x,u,\ldots,\nabla^{k}u\bigr)\, D^{\alpha}v \, dx,
%   $$
%   is pseudomonotone.
   
   In the present work, we combine and extend these two lines of research by investigating a \emph{generalized polyharmonic eigenvalue problem} in Orlicz--Sobolev spaces without the \(\Delta_2\)-condition, i.e., given a bounded domain $\Omega \subset \R^n$, we consider the eigenvalue problem

\begin{equation}\label{MP}
\begin{cases}
A(u) = \lambda h(x, u) & \text{in }\Omega\\
u =0  &  \text{on }\partial \Omega.
\end{cases}
\end{equation}

In order to apply the Ljusternik-Schnirelmann theory to \eqref{MP} we need to restrict ourselves to the case where $A$ is the derivative of a functional of the form 
$$\mathcal{G}(u):=\int_\Omega G(x, u, \nabla u, \dots, \nabla^k u)\,dx$$where $G: \Omega \times \mathbb{R}^N \to \mathbb{R}$ is a Carathéodory function with precise hypothesis given below. Therefore, we will consider operators of the form
$$\P^k_G(u) :=A(u)= \mathcal{G}'(u)= \sum_{\alpha\in A_k} (-1)^{|\alpha|} D^\alpha G_\alpha(x,\xi(u)),$$where $\xi(u)= ( u, \nabla u, \dots, \nabla^k u)$. The operator $\P^k_G$ will be called the $G$-polyharmonic operator of order $2k$.

%where the polyharmonic operator
%$$ \P^k_G(u) = \sum_{\alpha\in A_k} (-1)^{|\alpha|} D^\alpha G_\alpha(x,\xi(u)) $$ is the Fr\'echet derivative of the energy functional (see Section \ref{preliminaries} for notation)
%$$\mathcal{G}(u):=\int_\Omega G(x, u, \nabla u, \dots, \nabla^k u)\,dx$$where $G: \Omega \times \mathbb{R}^N \to \mathbb{R}$ is a Carathéodory function, strictly convex in the $\xi$-variable, and satisfying natural coercivity and growth conditions linked to an $N$-function $M$ that does not necessarily satisfy the $\Delta_2$-condition.  The function $\mathcal{H}'(u):= H(x, u) = \int_0^u h(x, t)\,dt$ represents a lower-order term with subcritical growth governed by another $N$-function $B$, chosen so that the embedding $W^k_0 E_M(\Omega) \hookrightarrow E_B(\Omega)$ is compact. Thus,  solutions of \eqref{MP} correspond to critical points of the Rayleigh-type quotient
%   $$
%   \mathcal{R}(u) = \frac{\displaystyle\int_\Omega G(x, u, \nabla u, \dots, \nabla^k u)\,dx}{\displaystyle\int_\Omega H(x, u)\,dx}.
%   $$
%   
%   
  Our main result can be summarized  as follows (for a precise statement, see Theorem \ref{main theorem}):
\begin{theorem}\label{main intro}Under the assumptions of $G$ and $H$ given in Section \ref{assumptions}, for any $\mu>0$, there exists a sequence of eigenvalues $\lambda_i$ of Problem \eqref{MP} such that $\lambda_i \to \infty$ as $i\to \infty$ and $\int_\Omega H(x, u_i)\,dx= \mu$, where $u_i$ is the eigenfunction associated to $\lambda_i$. 
\end{theorem}

 \subsection*{Organization of the paper}  The paper is organized as follows. In Section \ref{preliminaries}, we review the necessary background on Orlicz and Orlicz-Sobolev spaces, introduce the complementary system framework, and state our precise assumptions on $G$ and $H$. Section \ref{proof of theorem 1} contains the proof of the main existence theorem for eigenvalues and eigenfunctions.

%\newpage
%Motivar la relevancia y discutir antecedentes (laplaciano, bilaplaciano, poliarmónico, p-laplaciano, bi-p-laplaciano, Tienari, etc)
%
%
%El setting es Orlicz:
%
%Also, the function 

%The nonlinearity $H$ is a given function satisfying the conditions:
%
%
% 
% 
%  Recall that  the conjugate function $\bar{M}$ of $M$ is defined as
%$$\bar{M}(t):=\sup\left\lbrace st-M(s): s\in \mathbb{R}\right\rbrace.$$
%
%\begin{example}Let us consider the eigenvalue problem for the $m$-Laplacian
%\begin{equation}\label{eigenvalue g lap}
%\begin{cases}
%-\text{div}\left(m(|\nabla u|)\dfrac{\nabla u}{|\nabla u|}\right)=\lambda g(u) \quad \text{in }\Omega\\
%u \in W^1_0L_M(\Omega).
%\end{cases}
%\end{equation}Observe that \eqref{eigenvalue g lap} is a particular case of \eqref{general eigenvalue pro} taking $k=1$ and
%$$A_i(x, u, \nabla u) = m(|\nabla u|)\dfrac{\partial_i u}{|\nabla u|}, \quad i=1, ..., n.$$ The functions $A_\alpha$ satisfy (A1). To prove (A2), observe that since $m$ is increasing and so
%\begin{equation}\label{ex1}
%m(t)t \leq \int_{t}^{2t}m(s)\,ds \leq M(2t).
%\end{equation}
%Moreover, by Young's inequality,
%$$m(t)t= M(t)+ \bar{M}(m(t)),$$which implies
%$$ \bar{M}(m(t))=m(t)t -M(t) \leq M(2t),$$since $M(t)\geq 0$ for $t\geq 0$ and where we used \eqref{ex1}. Therefore, $(A2)$ holds. Finally, since $m$ is increasing, the assumption $(A3)$ is valid. 
%\end{example}
%
%The main result of the paper can be summarized as 
%\begin{theorem}
%There exists a sequence $\{\lambda_j\}_{j\in\N}$ of critical values of \eqref{general eigenvalue pro}
%\end{theorem}

\section{Preliminaries}\label{preliminaries}

\subsection{Orlicz and Orlicz-Sobolev spaces} 

This subsection contains no new material and it is well known for experts. It is included only for completeness and everything here is contained, for instance, in the book \cite{K}.

\subsubsection{$N-$functions}
We start with the definition of an $N$-function:

\begin{definition}
An $N$-function is a function $M\colon \R\to \R$ satisfying:
\begin{itemize}
\item $M$ is continuous, convex and even.
\item $M(t)>0$ for all $t>0$.
\item Finally, $M$ is sublinear at $0$ and superlinear at $\infty$, that is,
$$\lim_{t\to 0}\dfrac{M(t)}{t}=0 \quad \text{and }\quad \lim_{t\to \infty}\dfrac{M(t)}{t}=\infty.$$
\end{itemize}
\end{definition}

In \cite[Chapter 1]{K}, it is provided the following result.

\begin{theorem}Any $N$-function admits the representation
$$M(t)=\int_0^{|t|}m(s)\,ds$$where $m$ is  nondecreasing, left-continuous, positive for $s>0$ and it satisfies
$$m(0)=0 \quad\text{and}\quad\lim_{s\to \infty}m(s)=\infty.$$
\end{theorem}

The following order for $N$-functions will be used:
\begin{definition}
Given two $N$-functions $M_1$ and $M_2$, we write  $M_1 \ll M_2$ if 
$$
\lim_{t\to \infty}\dfrac{M_1(t)}{M_2(\lambda t)}=0,
$$
for any $\lambda>0$. 
\end{definition}

Given an $N-$function $M$, a fundamental tool in Orlicz spaces is the {\em conjugate} function $\bar M$.
\begin{definition}
Let $M$ be an $N-$function. The conjugate function $\bar M$ is defined as
$$
\bar M(s) = \sup\{st-M(t)\colon t>0\}.
$$
\end{definition}
Observe that $\bar M$ is the optimal function that satisfies the {\em Young-type} inequality:
$$
st\le M(t) + \bar M(s).
$$
It is a well known fact that $\bar M$ is also an $N$-function and that $\bar{\bar M} = M$.

An important class of $N$-functions is the $\Delta_2$-class.
\begin{definition}
An $N-$function $M$ is said to verity the $\Delta_2$-condition, if there exists $C>1$ such that
$$
M(2t)\le C M(t), \quad \text{for all } t>0.
$$
\end{definition}

By \cite[Theorem 4.1, Chapter 1]{K},  an N-function  satisfies the $\Delta_{2}$ condition if and only if there is $p^{+} > 0$ such that
			\begin{equation}\label{delta 2}
		\frac{tm(t)}{M(t)} \leq p^{+}, ~~~~~\forall\, t>0.
	\end{equation}
	On the other hand, by \cite[Theorem 4.3,   Chapter 1]{K}, a necessary and sufficient condition for the N-function $\bar{M} $  to satisfy the $\Delta_{2}$ condition is that there is $p^{-} > 1$ such that
			\begin{equation}\label{delta 2 2}
	p^{-} \leq 	\frac{tm(t)}{M(t)}, ~~~~~\forall\, t>0.
	\end{equation}

\begin{example}
Some examples of $N$-functions are
\begin{itemize}
\item Power functions: $M(t)=\frac{|t|^p}{p}$ with $1<p<\infty$

\item Logarithmic perturbations: $M(t) = \frac{|t|^p}{p}\ln^q(1+|t|)$, with $1<p<\infty$, $0<q<\infty$.

\item Exponential functions: $M(t) = e^{|t|}-1$
\end{itemize}

The first two examples verify the $\Delta_2$-condition while the third one does not.
\end{example}

\subsubsection{Orlicz spaces}
Let $\Omega$ be an open subset of $\R^n$. Given an $N$-function $M$, the Orlicz class $\mathcal{L}_M(\Omega)$ is defined as the set of real-valued and measurable functions $u$ so that
$$
\int_\Omega M(u(x))\,dx<\infty.
$$
The Orlicz class $\mathcal L_M(\Omega)$ is not in general a linear space, but it is always a convex space. The Orlicz space $L_M(\Omega)$ is defined as the linear hull of $\mathcal{L}_M(\Omega)$ and it is a Banach space equipped with the Luxemburg norm
$$
\|u\|_M=\inf\left\lbrace k>0\colon \int_\Omega M(\tfrac{u}{k})\,dx\le 1\right\rbrace.
$$
The closure in $L_M(\Omega)$ of the bounded and measurable functions with compact support in $\overline{\Omega}$ is denoted by $E_M(\Omega)$. We then have the inclusions
$$
E_M(\Omega)\subset \mathcal L_M(\Omega)\subset L_M(\Omega).
$$
These inclusions are strict unless $M$ satisfies the $\Delta_2$-condition in which we have all equalities.

On the other hand, for a general $N-$function $M$, $L_M(\Omega)$ is not separable. The separability of $L_M(\Omega)$ turns out to be equivalent to the $\Delta_2$-condition on $M$. For a general $N$-function $M$, $E_M(\Omega)$ is separable and its dual can be identified with $L_{\bar{M}}(\Omega)$ by means of the interior product
$$
\left\langle u, v \right\rangle :=\int_\Omega uv\,dx, \quad u \in E_M(\Omega), v \in L_{\bar{M}}(\Omega).
$$

Observe that $L_M(\Omega)$ is reflexive if and only if $M$ and $\bar M$ both satisfy the $\Delta_2$-condition (see \cite{K}), but for a general $N-$function $M$, we have the relations
$$
(E_M(\Omega))^* = L_{\bar M}(\Omega) \quad \text{and}\quad (E_{\bar M}(\Omega))^* = L_M(\Omega).
$$

\subsubsection{Orlicz-Sobolev spaces}
We introduce some not so standard notation on differentiation.  This notation has been introduced purely for clarity, since it has the advantage over the usual multi-index convention of making an explicit distinction in the order of the coordinate variables with respect to which differentiation is taken. This is specially useful in derivatives of general functionals (see Subsection \ref{notation}). However, we point out that by a rearrangement of derivatives, both notations are equivalent for smooth functions. 

Given $k\in \N$ and $\alpha = (\alpha_1,\dots,\alpha_k)\in \{1,\dots,n\}^k= I_n^k$ and $u\in C^\infty(\Omega)$, we denote
$$
D^\alpha u = \frac{\partial^k u}{\partial x_{\alpha_1}\cdots\partial x_{\alpha_k}}.
$$
We then denote the array
$$
\nabla^k u = (D^\alpha u\colon \alpha\in \{1,\dots,n\}^k)\subset  \R^{n^k}.
$$
Define the set of indices 
$$
A_k = \bigcup_{j=0}^k I_n^j,
$$
where $I_n^0=\{0\}$. Then the set of all derivatives of $u$ up to order $k$ is denoted by $\{ D^\alpha u\colon \alpha\in A_k\}$, where $D^0 u = u$.

We now introduce the definition of the Orlicz-Sobolev spaces.
\begin{definition}
Let $\Omega\subset\R^n$ be an open set, and $M$ be an $N-$function. The Orlicz-Sobolev space of order $k$, denoted by $W^kL_M(\Omega)$, is defined as the set of functions $u$ such that $u$ and its distributional derivatives up to order $k$ lie in $L_M(\Omega)$, i.e.
$$
W^kL_M(\Omega) = \{u\in L_M(\Omega)\colon D^\alpha u\in L_M(\Omega)\text{ for all }\alpha\in A_k\}.
$$
In this space we consider the norm
$$
\|u\|_{k, M}:= \sum_{\alpha \in A_k}\|D^\alpha u\|_M.
$$
\end{definition}
It is a well known fact that $(W^kL_M(\Omega), \|\cdot\|_{k, M})$ is a Banach space. Moreover, a natural and useful subspace is $W^kE_M(\Omega)$.
\begin{definition}
Let $\Omega\subset\R^n$ be an open set, and $M$ be an $N-$function. We define the space $W^kE_M(\Omega)$ as the set of functions $u$ such that $u$ and its distributional derivatives up to order $k$ lie in $E_M(\Omega)$, i.e.
$$
W^kE_M(\Omega) = \{u\in E_M(\Omega)\colon D^\alpha u\in E_M(\Omega)\text{ for all }\alpha\in A_k\}.
$$
\end{definition}

Clearly, $W^kE_M(\Omega)\subset W^kL_M(\Omega)$ and it is a closed subspace.

The spaces  $W^kL_M(\Omega)$ and $W^kE_M(\Omega)$ are naturally embedded into the product spaces $\Pi_{\alpha\in A_k}L_M(\Omega)$ and $\Pi_{\alpha\in A_k} E_M(\Omega)$ respectively. Hence, they inherit the functional properties of these product spaces, namely: $W^kE_M(\Omega)$ is separable and $W^kL_M(\Omega)$ is separable if and only if $M$ satisfies the $\Delta_2$-property, in which case both spaces agree.

In order to deal with boundary conditions we need to define the spaces of functions that vanish at the boundary.

\begin{definition}
Let $\Omega\subset \R^n$ be an open set and $M$ be an $N-$function. We define $W^k_0L_M(\Omega)$  as the closure of the set of test functions $\mathcal{D}(\Omega)$ with respect to the $\sigma(\Pi_{\alpha\in A_k}L_M(\Omega), \Pi_{\alpha\in A_k} E_{\bar M}(\Omega))$ topology and $W^k_0E_M(\Omega)$ as the closure of $\mathcal{D}(\Omega)$ with respect to the norm topology.
\end{definition}

For the spaces $W^k_0L_M(\Omega)$ and $W^k_0E_M(\Omega)$ the following Poincaré-type inequality was proved in  \cite[Lemma 5.7]{G}
\begin{equation}\label{Poincare}
\|u\|_{k, M}\le C \|\nabla^k u \|_M,
\end{equation}
for all $u\in W^k_0L_M(\Omega)$.

The dual space $W^{-k}L_{\bar M}(\Omega) = (W^k_0E_M(\Omega))^*$ was characterized in \cite{G}. In that paper it is shown that, if $\Omega$ has the segment property,
$$
W^{-k}L_{\bar{M}}(\Omega) = \left\lbrace f\in \mathcal{D}'(\Omega): f=\sum_{\alpha\in A_k}(-1)^{|\alpha|}D^\alpha f_\alpha,\quad f_\alpha \in L_{\bar{M}}(\Omega)\right\rbrace.
$$

Moreover it is also shown in \cite{G} that if we define
$$
W^{-k}E_{\bar{M}}(\Omega) = \left\lbrace f\in \mathcal{D}'(\Omega): f=\sum_{\alpha\in A_k}(-1)^{|\alpha|}D^\alpha f_\alpha,\quad f_\alpha \in E_{\bar{M}}(\Omega)\right\rbrace,
$$
then $(W^{-k}E_{\bar{M}}(\Omega))^* = W^k_0L_M(\Omega)$.

In other words, these facts prove that if we denote $Y=W^k_0L_M(\Omega)$, $Y_0=W^k_0E_M(\Omega)$, $Z=W^{-k}L_{\bar{M}}(\Omega)$ and $Z_0=W^{-k}E_{\bar{M}}(\Omega)$, then $(Y, Y_0, Z, Z_0)$ form a {\em complementary system}. That is, $Y$ and $Z$ are real Banach spaces in duality with respect to a continuous pairing and $Y_0$ and $Z_0$ are closed subspaces of $Y$ and $Z$ respectively such that, by means of the duality pairing, the dual of $Y_0$ can be identified to $Z$ and the dual of $Z_0$ can be identified to $Y$. See \cite{G, T}.

%Next, we introduce some embedding results from \cite{Cianci}. Given an $N$-function $M_0$ such that
%\begin{equation}
%\int_0^1 \dfrac{M_0^{-1}(t)}{t^{1+1/n}}\,dt <\infty \quad \text{and }\quad  \int_0^\infty \dfrac{M_0^{-1}(t)}{t^{1+1/n}}\,dt =\infty.
%\end{equation}We then define a new $N$ function by
%$$M_1^{-1}(t):= \int_0^t \dfrac{M_0^{-1}(s)}{s^{1+1/n}}\,ds.$$Repeating this procedure, we find a finite collection of $N$-functions $M_0, M_1, ..., M_q$, given by
%\begin{equation}\label{Mj}
%M_j^{-1}(t) := \int_0^t \dfrac{M_{j-1}^{-1}(s)}{s^{1+1/n}}\,ds,
%\end{equation}
%where $q\leq n$ is the largest integer such that
%$$\int_1^\infty \dfrac{M_{j-1}^{-1}(t)}{t^{1+1/n}}\,dt =\infty \quad \text{but }\int_1^\infty \dfrac{M_{j}^{-1}(t)}{t^{1+1/n}}\,dt <\infty.$$
%
%In [DT] it is proved the following immersion result:
%\begin{theorem}
%Let $\Omega$ be open and bounded. Let $M$ be an $N-$function and let $M_0, M_1,\dots, M_q$ be as in \eqref{Mj}. Then
%\begin{itemize}
%\item If $k\le q$ and $B\ll M_k$, then $W^k_0L_M(\Omega)\subset L_B(\Omega)$ with compact inclusion.
%
%\item If $k>q$, then $W^k_0L_M(\Omega)\subset C^{k-q-1}(\overline{\Omega})$, with compact inclusion.
%\end{itemize}
%\end{theorem}

In particular, in \cite{T} the following result for complementary systems is proved
\begin{theorem}\cite[Theorem $3.1$]{T}\label{Tienary}
Assume that $(Y, Y_0, Z, Z_0)$ is a complementary system, with $Y_0$ and $Z_0$ separable, the norm $\|\cdot\|_Z$ is dual to $\|\cdot \|_{Y_0}$, the norm $\|\cdot\|_{Y}$ is dual to $\|\cdot \|_{Z_0}$, and $V\subset Y_0$ is a norm-dense linear subspace. Then, there exists a sequence of mappings $P_n: Y_0\to Y_0$ satisfying:
\begin{itemize}
\item[(i)] $P_n$ is odd and norm-continuous for all $n$.
\item[(ii)] $P_n(Y_0)$ is contained in a finite-dimensional subspace of $V$ for all $n$.
\item[(iii)] If $\left\lbrace u_n\right\rbrace \subset Y_0$ and $u_n \to u \in Y$ for $\sigma(Y, Z_0)$, then $P_n(u_n)\to u$ for $\sigma(Y, Z_0)$.
\item[(iv)]If $\left\lbrace u_n \right\rbrace \subset Y_0$ and $u_n\to u\in Y$ strongly, then $\|P_n(u_n)\|_Y\to \|u\|_Y$.   
\end{itemize}
\end{theorem}

\subsection{Notation} \label{notation}

Let $$N = \sum_{j=0}^k n^j = \frac{n^{k+1}-1}{n-1},$$ then $\R^N\simeq \R\times \R^n\times \R^{n^2}\times \cdots\times \R^{n^k}$ and given $\xi\in \R^N$, we will write $\xi=(\xi^0,\xi^1,\dots,\xi^k)$, with $\xi^j\in \R^{n^j}$. For any index $\alpha\in A_k$, the order of $\alpha$ is defined as the unique $j$ such that $\alpha\in I_n^j$ and this order will be denoted by $|\alpha|$. 

With these notations, given $\xi\in \R^N$, its coordinates are given by
$$
\xi = (\xi^{|\alpha|}_\alpha)_{\alpha\in A_k}.
$$
In order to simplify the notation, we will just write $\xi_\alpha = \xi^{|\alpha|}_\alpha$.

Now, let $G\colon \Omega\times \R^N\to \R$, $G=G(x, \xi)$, be a Carathéodory function, i.e. for a.e. $x\in\Omega$, $G(x, \cdot)$ is continuous and for any $\xi\in \R^N$, $G(\cdot, \xi)$ is measurable.

Assume further that for a.e. $x\in\Omega$, $G(x, \cdot)$ is of class $C^1(\R^N)$. Then we denote its partial derivatives as
$$
G_\alpha(x, \xi) = \frac{\partial G}{\partial \xi_\alpha}(x,\xi), \quad\text{for any }\alpha\in A_k.
$$

With these notations, observe that 
$$
\frac{d}{d\ve}G(x,\xi + \ve\eta)|_{\ve=0} = \sum_{\alpha\in A_k} G_\alpha (x,\xi)\eta_\alpha,
$$
for any $\xi,\eta\in \R^N$ and a.e. $x\in\Omega$.

Finally, given $u\in C^\infty(\Omega)$, we denote $\xi(u) = (u, \nabla^1u,\dots, \nabla^ku)$ and observe that $\xi(u)_\alpha = D^\alpha u$.

Combining these notations, we easily obtain that
$$
\frac{d}{d\ve}\int_\Omega G(x, \xi(u) + \ve\xi(v))\, dx\Big|_{\ve=0} = \int_\Omega \sum_{\alpha\in A_k} G_\alpha(x, \xi(u)) D^\alpha v\, dx = \int_{\Omega} \sum_{\alpha\in A_k} (-1)^{|\alpha|} D^\alpha G_\alpha(x,\xi(u)) v\, dx,
$$
for every $v\in C^\infty_c(\Omega)$ where the last equality holds if $G(x,\cdot)$ is of class $C^2(\R^N)$ for a.e. $x\in\Omega$.

We will use the notation
\begin{equation}\label{polyharmonic op}
\P^k_G(u) = \sum_{\alpha\in A_k} (-1)^{|\alpha|} D^\alpha G_\alpha(x,\xi(u)) 
\end{equation}
and we will call it the $G-$polyharmonic operator.

\subsection{Assumptions}\label{assumptions}
Throughout this paper the following assumptions are made on the nonlinear functions $G$ and $H$:

\subsubsection{Assumptions on $G$}
\begin{itemize}
\item $G$ is Carathéodory, that is, for a.e. $x\in \Omega$, $G(x,\cdot)$ is continuous and for all $\xi\in \R^N$, $G(\cdot, \xi)$ is measurable.

\item $G(x, \cdot)$ is strictly convex a.e. $x\in \Omega$ and $G(x,0)=0$.

\item $G(x, \cdot)$ is of class $C^1$ a.e. $x\in\Omega$.

\item $G$ is $M-$coercive, that is, there exists a constant $\theta>0$ such that
\begin{equation}\label{M-coercive}
\theta \|\nabla^k u\|_M\le \int_\Omega G(x,\xi(u))\, dx .
\end{equation}
for any $u\in W^k_0L_M(\Omega)$.

\item For any $\alpha\in A_k$, the derivatives of $G$, $G_\alpha$, satisfy the growth condition
\begin{equation}\label{growth}
|G_\alpha(x, \xi)|\le a(x) + b\sum_{\beta\in A_k} \bar M^{-1}(M(c \xi_\beta)),
\end{equation}
where $a\in E_{\bar M}(\Omega)$ and $b, c>0$.
\end{itemize}

Some remarks are in order.
\begin{remark}\label{coercivity.Ga}
The $M-$coercivity condition on $G$ is equivalent to 
\begin{equation}\label{Gax}
\int_\Omega \sum_{\alpha\in A_k} G_\alpha(x, \xi(u))D^\alpha u\, dx \ge \theta \|\nabla^k u\|_M.
\end{equation}
\end{remark}
\begin{proof}
In fact, observe that since $G(x,\cdot) \in C^1$ and convex, a.e. $x\in\Omega$, then
$$
\frac{d}{dt}G(x, t\xi) = \sum_{\alpha\in A_k} G_\alpha(x,t\xi)\xi_\alpha
$$
is monotone in $t$ and so
\begin{equation}\label{Gcota}
G(x, \xi) = \int_0^1 \frac{d}{dt}G(x, t\xi)\, dt = \int_0^1 \sum_{\alpha\in A_k} G_\alpha(x,t\xi)\xi_\alpha\,dt \le \sum_{\alpha\in A_k} G_\alpha(x, \xi)\xi_\alpha.
\end{equation}

On the other hand, if \eqref{Gax} holds, 
$$
G(x, \xi(u)) = \int_0^1 \sum_{\alpha\in A_k} G_\alpha(x, t\xi(u))tD^\alpha u\, \frac{dt}{t}
$$
and integrating in $\Omega$,
$$
\int_\Omega G(x, \xi(u))\, dx = \int_0^1\int_\Omega \sum_{\alpha\in A_k} G_\alpha(x, t\xi(u))tD^\alpha u\, dx\, \frac{dt}{t}\ge \theta \int_0^1 \|\nabla^k (tu)\|_M\, \frac{dt}{t} = \theta \|\nabla^k u\|_M,
$$
as we wanted to show.
\end{proof}

\begin{remark}
Observe that the growth condition on $G_\alpha$ implies the bound
$$
G(x, \xi)\le \tilde a(x) + \tilde b \sum_{\beta\in A_k}M( c \xi_\beta),
$$
for some constant $\tilde b>0$ and $\tilde a\in L^1(\Omega)$.
\end{remark}

\begin{proof}
In fact, using Young's inequality, we observe that
\begin{equation}\label{cota.a}
a(x)|\xi_\alpha|\le \frac{1}{c}(\bar M(a(x)) +  M(c\xi_\alpha))
\end{equation}
and
\begin{equation}\label{cota.b}
\bar M^{-1}(M(c\xi_\beta))|\xi_\alpha|\le \frac{1}{c}(M(c\xi_\beta) + M(c\xi_\alpha)).
\end{equation}
Therefore, using the inequality 
$$
G(x,\xi)\le \sum_{\alpha\in A_k} G_\alpha(x,\xi)\xi_\alpha,
$$
proved in the previous remark, we get that
$$
G(x, \xi)\le  \sum_{\alpha\in A_k} |G_\alpha(x,\xi)| |\xi_\alpha| \le \sum_{\alpha\in A_k} \left(a(x) + b\sum_{\beta\in A_k} \bar M^{-1}(M(c\xi_\beta))\right)|\xi_\alpha|.
$$
This estimate together with \eqref{cota.a} and \eqref{cota.b} gives us the desired bound.
\end{proof}

\begin{remark}\label{monotonia derivadas}
It is standard to see that the strict convexity of $G$ implies the strict monotonicity of $(G_\alpha)_{\alpha\in A_k}$, that is
\begin{equation}\label{eq.monotonia}
\sum_{\alpha \in A_k} (G_\alpha(x, \xi) - G_\alpha(x, \eta))(\xi_\alpha - \eta_\alpha) \ge 0
\end{equation}
for any $\xi, \eta\in \R^N$, with equality if and only if $\xi = \eta$.
\end{remark}

\subsubsection{Assumptions on $H$}
On the source term $H(x, u)$ we assume the following:
\begin{itemize}
\item There exists a Carathéodory function $h(x, u)$ such that $h$ is odd in $u$ for a.e. $x\in\Omega$, $h(x, u)u>0$ for $u\neq 0$ and a.e. $x\in\Omega$ and 
$$
H(x, u) = \int_0^u h(x, t)\, dt.
$$

\item $h(x, u)$ satisfies the growth condition
\begin{equation}\label{b.growth}
|h(x, u)|\le f(x) + c\bar B^{-1}(B(cu)),
\end{equation}
for some $f \in L_{\bar B}(\Omega)$ and a subcritical $N-$function $B$, in the sense that the immersion $W^k_0E_M(\Omega)\subset E_B(\Omega)$ is compact.
\end{itemize}

\begin{remark}
For sharp conditions on $B$ and $M$ such that the compactness of the immersion $W^k_0E_M(\Omega)\subset E_B(\Omega)$ holds, we refer to \cite{Cianchi}. See also \cite{DT} for earlier results.
\end{remark}

\begin{remark}
Arguing as before, the growth condition on $h$ implies 
\begin{equation}\label{B.growth}
|H(x, u)|\le \tilde f(x) + \tilde c B(\tilde c u),
\end{equation}
for some $\tilde f \in L^1(\Omega)$ and some constant $\tilde c>0$.
\end{remark}

\section{A generalized eigenvalue problem}\label{proof of theorem 1}

In what follows we will occasionally use this notation  $Y=W^k_0L_M(\Omega)$, $Y_0=W^k_0E_M(\Omega)$, $Z=W^{-k}L_{\bar{M}}(\Omega)$ and $Z_0=W^{-k}E_{\bar{M}}(\Omega)$, and recall that $(Y, Y_0, Z, Z_0)$ is a complementary system.

We proceed as in \cite{T}, and so we take $\left\lbrace v_1, v_2, ...\right\rbrace \subset \mathcal{D}(\Omega)$ to be a countable norm-dense linearly independent subset of $Y_0$ and let $\left\lbrace P_n \right\rbrace$ be the sequence of mappings from Theorem \ref{Tienary} with
$$V_n= \text{span}\left\lbrace v_1, ..., v_n\right\rbrace.$$According to Theorem \ref{Tienary}, for each $n$, there is a positive integer $m_n$ so that
$$P_n(Y_0)\subset V_{m_n}.$$
Define the functionals $\G:D_\G\to \mathbb{R}$ and $\H:D_\H\to \mathbb{R}$ as
$$
\G(u):= \int_\Omega G(x, u, \nabla u, ..., \nabla^k u)\,dx\quad\text{and}\quad \H(u):=\int_\Omega H(x,u)\,dx.
$$
Here, $D_\G=\left\lbrace u\in Y: \G(u)<\infty\right\rbrace$ and $D_\H=\left\lbrace u\in Y: \H(u)<\infty\right\rbrace$. It is easy to see that $Y_0\subset D_\G\subset Y$ and $Y_0\subset D_\H\subset Y$.

Observe that $\G$ and $\H$ are not differentiable functionals in their domains unless $M$ and $B$ satisfy the $\Delta_2$-condition respectively. However, for $u, v\in Y_0$ we have that
\begin{equation}\label{domG'}
\langle v, \G'(u)\rangle = \int_\Omega \sum_{\alpha \in A_k} G_\alpha(x, \xi(u))D^\alpha v\, dx <\infty,
\end{equation}
so $\G'(u)=\P^k_G(u)$ (recall \eqref{polyharmonic op}). We then define $D_{\G'} = D_{\P^k_G}$ as the functions $u\in Y$ such that \eqref{domG'} holds for every $v\in Y_0$.

In a similar manner, we define $D_{\H'}$ as the functions $u\in Y$ such that
$$
\int_\Omega h(x,u)v\, dx<\infty
$$
for all $v\in Y_0$.

Under our assumptions on $G$, the $G-$polyharmonic operator $\P^k_G$ fits into the theory developed by Gossez in \cite{G}. In particular it is a {\em pseudo monotone} operator, that is:
\begin{theorem}[\cite{G}, Theorem 4.1]\label{pseudo}
Let $G$ satisfy the assumptions from Section \eqref{assumptions}. Then $\P^k_G$ is a pseudo-monotone operator, that is for any sequence $\{u_n\}_{n\in\N}\in D_{\P^k_G}$ such that 
\begin{itemize}
\item $u_n\to u$ in the $\sigma(W^k_0L_M(\Omega), W^{-k}E_{\bar M}(\Omega))$ topology, 
\item $\P^k_G(u_n)\to \chi$ in the $\sigma(W^{-k}L_{\bar M}(\Omega), W^k_0 E_M(\Omega))$ topology,
\item $\limsup_{n\to\infty} \langle u_n, \P^k_G(u_n)\rangle \le \langle u, \chi\rangle$,
\end{itemize}
imply that
\begin{itemize}
\item $u\in D_{\P^k_G}$,
\item $\chi = \P^k_G(u)$,
\item $\lim_{n\to\infty} \langle u_n, \P^k_G(u_n)\rangle = \langle u, \chi\rangle$.
\end{itemize}
\end{theorem}

Now, observe that if we restrict ourselves to the finite dimensional subspace $V_n$,
$$
\left\langle v, \P^k_G(u) \right\rangle = \int_\Omega \sum_{\alpha\in A_k} G_\alpha(x, \xi(u)) D^\alpha v\, dx \quad \text{and}\quad \left\langle v, \H'(u) \right\rangle = \int_\Omega h(x, u)v\,dx,
$$
for all $u, v\in V_n$. 

Observe that the strict monotonicity of $G_\alpha$, \eqref{eq.monotonia}, implies 
$$
\langle u, \P^k_G(u)\rangle = \int_\Omega \sum_{\alpha\in A_k} G_\alpha(x, \xi(u))D^\alpha u\, dx>0
$$
if $u\neq 0$, $u\in V_n$.

In order to apply the Ljusternik-Shnirelman theory, again as in \cite{T}, we introduce some notations.
\begin{align*}
\mathcal M_r &= \{u\in Y_0\colon \G(u)=r\},\\
\mathcal K_i(r) &= \{K\subset \mathcal M_r \text{ compact, symmetric}\colon \text{gen }K\ge i\},\\
\mathcal K_{i,n}(r) &= \{K\subset \mathcal M_r\cap V_n \text{ compact, symmetric}\colon \text{gen }K\ge i\},\\
c_i(r) &= \sup_{K\in \mathcal K_i(r)} \inf_{u\in K} \H(u),\\
c_{i,n}(r) &= \sup_{K\in \mathcal K_{i,n}(r)} \inf_{u\in K} \H(u).
\end{align*}

We can then apply the Ljusternik-Shnirelman theory on finite dimensional spaces to obtain the following lemma.
\begin{lemma}
Let $n\in \N$ be fixed. Then there exist $u_1^n,\dots, u_n^n\in V_n$ and $\lambda_1^n,\dots, \lambda_n^n>0$, such that
\begin{align*}
\G(u_i^n)=r,\quad \H(u_i^n) = c_{i,n}(r)\quad\text{and}\quad \P^k_G(u_i^n) = \lambda_i^n \H'(u_i^n) \text{ in } V_n^*,
\end{align*}
for $i=1,\dots, n$.
\end{lemma}

\begin{proof}
With the properties proved for our functionals, the proof follows without change from that of \cite[Lemma 4.1]{T}.
\end{proof}

Next, for any given $i\in \N$ we analyze the limiting behavior for the sequences $\{\lambda_i^n\}_{n\ge i}$ and $\{u_i^n\}_{n\ge i}$.

\begin{lemma}\label{bar u lemma}
There exist $\bar \lambda\in (0,\infty)$ and $\bar u\in W^k_0L_M(\Omega)\cap E_B(\Omega)$ such that, up to a subsequence, $\lambda_i^n\to\bar\lambda$ as $n\to\infty$ and $u_i^n\rightharpoonup \bar u$ in $\sigma(Y, Z_0)$. Moreover,
\begin{align*}
&\G(\bar u)=r,\quad \H(\bar u) = \lim_{n\to\infty} c_{i,n}(r),\quad \bar u \in D_{P^k_G}\cap D_{\H'}\quad \text{and}\quad \P^k_G(\bar u) = \bar \lambda \H'(\bar u).
\end{align*}
That is,  for every $v \in Y_0$,
$$
\int_\Omega \sum_{\alpha\in A_k} G_\alpha(x, \xi(\bar u))D^\alpha v \, dx = \bar \lambda\int_\Omega h(x, \bar u) v\, dx,
$$
\end{lemma}

\begin{proof}
Observe that by the $M-$coercivity condition \eqref{M-coercive}, since $\{u_i^n\}_{n\ge i}\subset \mathcal M_r$, and by Poincaré inequality \eqref{Poincare}, we have that the sequence $\{u_i^n\}_{n\ge i}$ is bounded in $Y_0\subset Y$. Hence, there exists $\bar u\in Y$ such that, up to a subsequence, $u_i^n\to \bar u$ as $n\to\infty$ in the $\sigma(Y, Z_0)$ topology.

By our assumptions, the embedding $Y_0=W^k_0E_M(\Omega)\subset E_B(\Omega)$ is compact and therefore we can suppose that $u_i^n\to \bar u$ strongly in $E_B(\Omega)$ and pointwise a.e. in $\Omega$ (consequently, $u\in E_B(\Omega)$). In particular, $B(u_i^n)\to B(\bar u)$ in $L^1(\Omega)$ (by the Brezis-Lieb Lemma \cite{BL}). Combining this fact with the growth condition on $H$, \eqref{B.growth}, we obtain 
$$
\H(u_i^n)\to \H(\bar u) \text{ as } n\to\infty.
$$

Next we want to show that the sequence of eigenvalues $\{\lambda_i^n\}_{n\ge i}$ is bounded. 

Since $c_{i,n}(r)$ is nondecreasing in $n$, we easily conclude that $\H(\bar u)\neq 0$ and this implies that $H(x, \bar u)\neq 0$, therefore $\bar u\neq 0$ a.e. in $\Omega$ and so $h(x, \bar u)\neq 0$ a.e. in $\Omega$.

Using the fact that $\bigcup_{n\ge i} V_n$ is norm dense in $Y$, we can find $n_0\ge i$ and $\phi\in V_{n_0}$ such that
$$
\int_\Omega h(x, \bar u)(\bar u-\phi)\, dx <0.
$$
Then we use the monotonicity of the operator $\P^k_G$ to conclude that
$$
0\le \langle u_i^n-\phi, \P^k_G(u_i^n) - \P^k_G(\phi)\rangle = \lambda_i^n \langle u_i^n-\phi, \H' (u_i^n)\rangle - \langle u_i^n-\phi,  \P^k_G (\phi)\rangle.
$$
Therefore,
$$
\lambda_i^n \le \frac{\langle u_i^n-\phi,  \P^k_G (\phi)\rangle}{\langle u_i^n-\phi, \H' (u_i^n)\rangle}.
$$
Observe that arguing exactly as before, the growth condition on $h$, \eqref{b.growth}, implies that $\H'(u_i^n) \to \H'(\bar u)$ in the $\sigma(Z_0, Y)$ topology. Hence
$$
\limsup_{n\to\infty}\lambda_i^n \le \frac{\langle \bar u-\phi,  \P^k_G (\phi)\rangle}{\langle \bar u-\phi, \H' (\bar u)\rangle}<\infty.
$$

We can then assume that $\lambda_i^n\to \bar\lambda$ as $n\to\infty$ and so
\begin{equation}\label{limit}
\lim_{n\to\infty} \langle u_i^n, \P^k_G (u_i^n)\rangle = \lim_{n\to\infty} \lambda_i^n\langle u_i^n, \H' (u_i^n)\rangle = \bar\lambda\langle \bar u, \H'(\bar u)\rangle.
\end{equation}

We now claim that this implies that $G_\alpha(x, \xi(u_i^n))$ is bounded in $L_{\bar M}(\Omega)$. Indeed, by the growth condition \eqref{growth} and convexity of $\bar{M}$,
\begin{align}\label{bound in bar M}
\bar M(\tfrac{1}{\delta}G_\alpha(x,\xi))&\le \bar M\left(\tfrac{a(x)}{\delta} + \tfrac{b}{\delta}\sum_{\beta\in A_k} \bar M^{-1}(M(c\xi_\beta))\right)\\
&\le \frac{1}{\delta}\bar M(a(x)) + \tfrac{b}{\delta} \sum_{\beta\in A_k} M(c\xi_\beta),
\end{align}
where $\delta>0$ is given by $\frac{1}{\delta} + (\# A_k)\frac{b}{\delta}=1$.

By a Poincaré-type inequality \eqref{Poincare}, we have that
$$
\|D^\beta u_i^n\|_M\le C\|\nabla^k u_i^n\|_M \le c\qquad \text{for all }\beta\in A_k,\ n\ge i.
$$
Hence, since  $\{u_i^n\}_{n\ge i} \subset Y_0$, from \eqref{bound in bar M} it follows that $G_\alpha(x, \xi(u_i^n))$ is bounded in $L_{\bar M}(\Omega)$. Hence, we may assume that 
$$
\P^k_G(u_i^n)\to \chi \in Z \qquad \text{ in } \sigma(Z, Y_0).
$$
Let us now show that $\chi = \P^k_G(\bar u)$. To this end, let us first take $v\in Y_0$ and compute
$$
\langle v, \chi\rangle = \lim_{n\to\infty} \langle v, \P^k_G(u_i^n)\rangle = \lim_{n\to\infty} \lambda_n\langle v, \H'(u_i^n)\rangle = \lambda\int_\Omega h(x, \bar u) v\, dx.
$$
Arguing exactly as in \cite[Lemma 4.2]{T}, this equality holds for any $v\in Y$, in particular for $v=\bar u$.

Hence,
\begin{align*}
\lim_{n\to\infty} \langle u_i^n, \P^k_G(u_i^n)\rangle &= \lim_{n\to\infty} \lambda_n\langle u_i^n, \H'(u_i^n)\rangle = \lambda \int_\Omega h(x, \bar u)\bar u\, dx = \langle \bar u, \chi\rangle.
\end{align*}
So, by the pseudomonotonicity of $\P^k_G$, Theorem \ref{pseudo}, it follows that $\bar u\in D_{\P^k_G}$, $\P^k_G(\bar u) = \chi$.

Since 
\begin{align*}
\int_\Omega \sum_{\alpha\in A_k} G_\alpha(x, \xi(u_i^n))D^\alpha u_i^n\, dx &= \lambda_n \int_\Omega h(x, u_i^n)u_i^n\, dx,\\
\int_\Omega \sum_{\alpha\in A_k} G_\alpha(x, \xi(\bar u))D^\alpha \bar u\, dx &= \bar\lambda \int_\Omega h(x, \bar u)\bar u\, dx
\end{align*}
and $\lambda_n\to\bar \lambda$, the growth condition on $h$, \eqref{b.growth}, and the fact that $H(x, u_i^n)$ converges to $H(x, \bar u)$ in $L^1(\Omega)$, tells us that $\sum_{\alpha\in A_k} G_\alpha(x, \xi(u_i^n))D^\alpha u_i^n\to \sum_{\alpha\in A_k} G_\alpha(x, \xi(\bar u))D^\alpha \bar u$ in $L^1(\Omega)$. From this follows that there exists a majorant $\psi\in L^1(\Omega)$ such that
$$
G(x, u_i^n)\le \sum_{\alpha\in A_k} G_\alpha(x, \xi(u_i^n))D^\alpha u_i^n\le \psi.
$$
Hence
$$
r = \G(u_i^n)\to \G(\bar u).
$$
The proof is now complete.
\end{proof}

The following lemma concerns the convergences of the sequences $\left\lbrace c_{i, n}(r)\right\rbrace$ as $n\to \infty$ and $\left\lbrace c_i(r)\right\rbrace$ as $i\to \infty$, respectively. The proof is the same as in \cite{T} and appeals to Theorem \ref{Tienary}.

\begin{lemma}\label{lemma convergences}
For each $i\in \mathbb{N}$, $c_{i, n}(r)\to c_i(r)$ as $n\to \infty$. Moreover, $c_i(r)\to 0$ as $i\to \infty$.
\end{lemma}

Now, we can precisely state the main result of the paper, namely Theorem \ref{main intro}. More precisely, we obtain the following

 \begin{theorem}\label{main theorem}
Let $\Omega\subset \mathbb{R}^n$ be open and bounded with the segment property and let $r>0$. Assume that $G$ and $H$ satisfy the assumptions from Section \ref{assumptions}. Then, there exist sequences $\left\lbrace \bar{u}_i\right\rbrace_{i=1}^\infty \subset W_0^{k}L_M(\Omega)$ and $\left\lbrace \bar{\lambda}_i\right\rbrace_{i=1}^\infty \subset (0, \infty)$ such that
\begin{equation}\label{equation bar u}
\mathcal{P}_G^k (\bar{u}_i) = \bar{\lambda}_i \H'(\bar{u}_i) \quad \text{in } W^{-k}L_{\bar{M}}(\Omega),
\end{equation}
and
\begin{equation}\label{restriction bar u}
\mathcal{G}(\bar{u}_i)=r, \quad \mathcal{H}(\bar{u}_i)= c_i(r),
\end{equation}
for all $i$. Moreover, $\bar{\lambda}_i\to \infty$ as $i\to \infty$ and $\bar{u}_i\to 0$ in the weak topology $\sigma( W_0^{k}L_M(\Omega),  W^{-k}E_{\bar{M}}(\Omega))$ as $i\to \infty$.  
\end{theorem}

The proof of this result follows the lines of  \cite[Theorem 4.5]{T}, but we provide it for completeness.

\begin{proof}
By Lemma \ref{bar u lemma}, for each $i\in \mathbb{N}$, there are $\bar u_i \in D_{\mathcal{P}_G^k}$ and $\bar{ \lambda}_i>0$ such that \eqref{equation bar u} and \eqref{restriction bar u} hold. 

Moreover, by Lemma \ref{lemma convergences} and the definition of $c_i(r)$, we get $\mathcal{G} (\bar{u}_i)\to 0$ as $i \to \infty$.  Hence,  the $M$-coercivity of $G$ \eqref{M-coercive} together with the Poincar\'e-type inequality \eqref{Poincare} imply that
\begin{equation}\label{conv to 0}
\|\bar u_i\|_M \to 0 \quad \text{as }i\to \infty.
\end{equation}

By \eqref{equation bar u} and appealing to \eqref{Gcota}, we get 
\begin{equation}\label{eigenvalue}
\bar{\lambda}_i= \dfrac{\left\langle \bar{u}_i, \mathcal{P}_G^k(\bar{u}_i) \right\rangle}{\left\langle \bar{u}_i, \mathcal{H}'(\bar{u}_i)\right\rangle}\geq \dfrac{\mathcal{G}(\bar{u}_i)}{\left\langle \bar{u}_i, \mathcal{H}'(\bar{u}_i)\right\rangle}= \dfrac{r}{\left\langle \bar{u}_i, \mathcal{H}'(\bar{u}_i)\right\rangle}.
\end{equation}
Finally, by the compactness of the embedding $W_0^mL_M(\Omega)\subset L_B(\Omega)$ and the fact that $\bar{u}_i\in E_B(\Omega)$, we have
 $\left\langle \bar{u}_i, \mathcal{H}'(\bar{u}_i)\right\rangle \to 0$ as $i\to \infty$, and so by \eqref{eigenvalue}, we obtain $\bar{\lambda}_i\to \infty$ as $i\to \infty$. This ends the proof of the theorem.
\end{proof}

\section*{Acknowledgments}

P. O.  was partially supported by CONICET PIP 11220210100238CO and SIIP 80020240400005UN . This work was started when the third author was visiting the Instituto de C\'alculo, at  Universidad de Buenos Aires. He would like to thank the members of the IC-UBA for their hospitality.

\end{document}